\documentclass[11pt,a4paper]{amsart}
\usepackage{amssymb}
\usepackage{amsthm}
\usepackage{amsmath}
\usepackage{amsxtra}
\usepackage{latexsym}
\usepackage{mathrsfs}
\usepackage[all,cmtip]{xy}
\usepackage[all]{xy}
\usepackage{enumitem}
\usepackage{xcolor}
\usepackage{graphicx}
\usepackage{comment}
\usepackage{mathabx,epsfig}
\usepackage{stmaryrd}

\usepackage{hyperref}

\newtheorem{thm}{Theorem}[section]
\newtheorem{lem}[thm]{Lemma}

\newtheorem{claim}[thm]{Claim}

\newtheorem{cor}[thm]{Corollary}

\theoremstyle{definition}

\newtheorem{rmk}[thm]{Remark}

\numberwithin{equation}{section}

\newcommand\be{\begin{equation}}
\newcommand\ba{\begin{eqnarray}}
\newcommand\ee{\end{equation}}
\newcommand\ea{\end{eqnarray}}
\newtheoremstyle{case}{}{}{}{}{}{:}{ }{}
\theoremstyle{case}

\def\C{{\mathbb C}}

\def\R{{\mathbb R}}
\def\Z{{\mathbb Z}}
\def\P{{\mathbb P}}

\def\N'{\lceil \max_{f \in S}\deg(f) (N' +1) \rceil}

\title[]
{Dynamical Cancellation of Polynomials }

\author{Xiao Zhong}
\thanks{The author was supported in part by NSERC grant RGPIN-2022-02951.}
\address{University of Waterloo \\
Department of Pure Mathematics \\
Waterloo, Ontario \\
Canada  N2L 3G1}
\email{x48zhong@uwaterloo.ca}

\begin{document}
\keywords{Dynamical cancellation, invariant curves, arithmetic dynamics, semigroups of maps, rational points on curves}
\subjclass[2020]{ 37P55, 14G05}
\maketitle
\begin{abstract}
    Extending the work of Bell, Matsuzawa and Satriano, we consider a finite set of polynomials $S$ over a number field $K$ and give a necessary and sufficient condition for the existence of a $N \in \mathbb{N}_{> 0}$ and a finite set $Z \subset \P^1_K \times \P^1_K$ such that for any $(a,b) \in (\P^1_K \times \P^1_K) \setminus Z$ we have the cancellation result:
   if $k>N$ and $\phi_1,\ldots ,\phi_k$ are maps in $S$ such that $\phi_{k} \circ \dots \circ \phi_1 (a) = \phi_k \circ \dots \circ \phi_1(b)$, then in fact
    $\phi_N \circ \dots \circ \phi_1(a) = \phi_N \circ \dots \circ \phi_1(b)$. Moreover, the conditions we give for this cancellation result to hold can be checked by a finite number of computations from the given set of polynomials.

\end{abstract}

\section{Introduction}
Let $X$ be a projective variety and let $f : X \to X$ be a surjective self-map, with both $X$ and $f$ defined over a number field $K$. An important part of understanding the dynamics of $(X,f)$ is to find the closed subvarieties $Y \subseteq X$ that are invariant under $f$; i.e., the varieties $Y$ with $f(Y) \subseteq Y$. This naturally leads one to look at the iterated preimages of $Y$ under $f$. A series of observations on the increasing geometric complexity of the iterated preimages and how, in principal, the geometric structure should exert significant influence on the underlying arithmetic structure gives rise to the expectation that the tower of $K$-points 
\[
Y(K) \subseteq (f^{-1}(Y))(K) \subseteq (f^{-2}(Y))(K) \subseteq \cdots
\]
should eventually stabilize. This expectation has been made precise as the Preimages Question in \cite[Question 8.4 (1)]{Preimages} and solved with an affirmative answer when $X$ is a smooth variety with nonnegative Kodaira dimension and $f$ is \'etale \cite[Theorem 1.2]{DynamicalCrevised}.

A special case of the Preimages Question is when $f$ is a coordinate-wise map $(g,g)$ on $X \times X$, where $g: X \to X$ is a surjective self-morphism. In this case, the Preimages Questions is translated to a cancellation problem; i.e. it says that there should exist a natural number $s$ such that for all $x,y \in X(K)$, if $g^n(x) = g^n(y)$ for some $n > s$ then we have $g^s(x) = g^s(y)$. The precise statement is given by \cite[Conjecture 1.5]{DynamicalCrevised} and in the same paper \cite[Theorem 1.3]{DynamicalCrevised} the case when $X$ is a curve is proved by applying $p$-adic uniformization techniques of Rivera-Letelier \cite{Uniformization}. Since significantly more is known about polynomial self-maps on $\mathbb{P}^1$ than surjective self-morphisms of projective varieties in general, the authors also prove a more general cancellation result allowing multiple polynomial self-maps on $\mathbb{P}^1$ as follows.

Henceforth we let $P_r(x)$ and $T_r(x)$ denote respectively the cyclic polynomial, $x^r$, of degree $r$ and the Chebyshev polynomial of degree $r$.
\begin{thm}\cite[Theorem 1.7]{DynamicalCrevised}
Let $K$ be a number field and let $\phi_1, 
\dots, \phi_r$ be polynomial maps on $\P^1_K$ of degree at least two. Suppose that none of the indecomposable factors of $(\phi_i)_{\overline{K}}$ are linearly related to a $T_d$ with $d$ odd or a $P_m$ for some $m$. Then there is a finite set $Z \subset (\P^1_K \times \P^1_K)(K)$ with the following property: for any point $a$, $b \in \P^1_K$ such that $(a,b) \notin Z$, if 
\[\phi_{i_n} \circ \dots \circ \phi_{i_1} (a) = \phi_{i_n} \circ \dots \circ \phi_{i_1} (b)\]
for some $n \geq 0$ and $i_1, \dots, i_n \in \{1, \dots, r\}$, then 
\[\phi_{i_2} \circ \phi_{i_1}(a) = \phi_{i_2} \circ \phi_{i_1}(b).\]
\end{thm}
 This theorem, however, is not optimal. For example this theorem does not say anything when $S = \{ T_5, P_5 \}$ although a cancellation result can be shown to hold in this case. In this paper we completely characterize when we have a dynamical cancellation result involving multiple polynomials by proving a necessary and sufficient statement in terms of the polynomials in $S$.

\begin{thm}\label{main:theorem}
Let $S$ be a finite set of polynomials of degree at least 2 defined over a number field $K$ and let $\langle S \rangle$ be the monoid generated by $S$ under composition. Then there exists a $N \in \mathbb{N}^+$ and a finite set of points $Z \subset \mathbb{P}^1_K \times \mathbb{P}^1_K $ such that if 
\begin{equation}\label{canceleq1}
    \phi_{k} \circ \dots \circ \phi_{1} (a) = \phi_{k} \circ \dots \circ \phi_{1}(b)
\end{equation}

with $\phi_{j} \in S$, $k > N$ and $(a,b) \in (\mathbb{P}^1_K \times \mathbb{P}^1_K) \setminus Z$ then we have 
\begin{equation}\label{canceleq2}
    \phi_{N} \circ \dots \circ  \phi_{1}(a) = \phi_{N} \circ \dots \circ \phi_{1}(b)
\end{equation}

unless there exist 
$h_1,h_2 \in \langle S \rangle$, where $h_1$, $h_2$ can be written as a composition of no more than $N$ elements in $S$ and a complex linear polynomial $l(x) = ux + v$, $u \in \C^*$ and $v \in \C$, such that one of the following holds:
\begin{enumerate}
    \item $h_2 \circ l = P \circ P_d$ and $l^{-1} \circ h_1 \circ l(x) = x Q(x^d)$, where $d \in \mathbb{N}_{\geq 2}$, $u = 1, v\in K$ and $K$ contains a primitive $d$-th root of unity;
    \item $  h_2 \circ l = P \circ P_d$ and $l \circ h_1^{-1} \circ l = P_r$, where $d| (r-1)$, $d \geq 2$ and $v \in K$ and $K$ contains a primitive $d$-th root of unity;

    \item $ h_2 \circ l = P \circ  T_d$ and $l^{-1} \circ h_1  \circ l = \pm T_r$, where $d | (r + 1)$ or $d | (r - 1)$ and one of the following holds:
    \begin{enumerate}
        \item $d =  2$ and $v \in K$;
        \item $d > 2$ and there exists a primitive $d$-th root of unity $\epsilon$ such that the conic 
    $$
        V(Y^2 + (\epsilon + 1/\epsilon -2)v\cdot (X + Y) - (\epsilon + 1/\epsilon)XY + X^2 + (2- \epsilon - 1/\epsilon)v^2 + (\epsilon - 1/\epsilon)^2u^2)
    $$
      has a $K$-point; Moreover, we can further refine the condition if $r$ is even. In this case the conic having a $K$-point is equivalent to $u, v, \epsilon + 1/\epsilon \in K$.
    \end{enumerate}

\end{enumerate}
 Moreover, if there exist $h_1$, $h_2$ and $l$ as above such that one of $(1)-(3)$ holds, then for any $j \in \mathbb{N}_{\geq 0}$, there exists an infinite set of $K$-points $(a,b)$ in $\P^1_K \times \P^1_K$ such that 
 \begin{equation}\label{eq: violate-cancel1}
     h_2 \circ h^{j}_1(a) = h_2 \circ h^{j}_1(b)
 \end{equation}
\begin{equation} \label{eq: violate-cancel2}
    h^{\circ j}_1(a) \neq h_1^{\circ j}(b).
\end{equation}

\end{thm}
\begin{rmk}
      Given the number field $K$ and the conic 
      $$
      C =  V(Y^2 + (\epsilon + 1/\epsilon -2)v\cdot (X + Y) + (\epsilon + 1/\epsilon)XY + X^2 + (2 - \epsilon - 1/\epsilon)v^2 + (\epsilon - 1/\epsilon)^2u^2),
      $$
      there are standard algorithms implemented in Sage to check if $C$ has $K$-points. So the condition of case $(3)$ in Theorem \ref{main:theorem} can be checked with a finite number of computations.
  \end{rmk}
 We will also show that the number $N$ is effectively computable in terms of the set $S$ of polynomials. Therefore the three cases in Theorem \ref{main:theorem} can be checked with a finite number of computations. 
\section{Preliminary Lemmas}
In this section we provide some preliminary lemmas for the proof of the Theorem \ref{main:theorem}.

\begin{lem} \label{pullb}
Let $\pi(x) = x + 1/x $. If $\epsilon$ and $\epsilon'$ are roots of unity such that $V(Y - \epsilon X)$, $V(Y - \epsilon' X) \subseteq \overline{(\pi, \pi)^{-1}(C)}$ for some irreducible curve $C \subseteq \mathbb{P}^1_\mathbb{C} \times \mathbb{P}^1_\mathbb{C}$ then either $\epsilon = \epsilon'$ or $\epsilon = \epsilon'^{-1}$. Moreover, $V(Y - \epsilon X) \subseteq \overline{(\pi, \pi)^{-1}(C)}$ if and only if $V(Y - \epsilon^{-1}X)\subseteq \overline{(\pi, \pi)^{-1}(C)}$.
\end{lem}
\begin{proof}
Notice that 
$$(\pi, \pi)(V(Y - \epsilon X)\cap (\C^* \times \C^*)) = \{(x + 1/x , \epsilon^{-1}/x + x / \epsilon^{-1}) : x \in \C^* \}$$
is Zariski constructible by Chevalley's theorem \cite[1.8.4]{EGA} as $$V(Y - \epsilon X) \cap (\C^* \times \C^*)$$ is Zariski constructible in $\C^* \times \C^*$ and $(\pi, \pi): \C^* \times \C^* \to \P^1_\C \times \P^1_\C$ is a regular morphism of varieties. The same is true for the image of $$V(Y - \epsilon' X)\cap (\C^* \times \C^*)$$ under $(\pi, \pi)$. By the condition in the statement of the lemma, the Zariski closure of their images are both $C$ as $C$ is irreducible. Then there exists subsets 
$$U \subseteq (\pi, \pi)(V(Y - \epsilon X)\cap (\C^* \times \C^*))$$
and $$U' \subseteq (\pi, \pi)(V(Y - \epsilon' X)\cap (\C^* \times \C^*))$$
such that $U$ and $U'$ are open and dense in $C$ since a constructible set of an irreducible curve is either finite or cofinite and in this case, $$(\pi, \pi)(V(Y - \epsilon' X)\cap (\C^* \times \C^*))$$ is cofinite. Therefore $U \cap U'$ is an infinite set. Thus there exists infinitely many $x_0 \in \C^*$ such that $(x_0 + 1/x_0, \epsilon^{-1} / x_0 + x_0 / \epsilon^{-1}) \in U \cap U'$. 
Hence for each such $x_0$, there exists $y_0 \in \C^*$ such that 
\[(x_0 + 1/x_0, \epsilon^{-1} / x_0 + x_0 / \epsilon^{-1}) = (y_0 + 1/y_0, \epsilon'^{-1} / y_0 + y_0 / \epsilon'^{-1}).\]
Therefore, $y_0 = x_0$ or $y_0 = 1/x_0$ from $x_0 + 1/x_0 = y_0 + 1/y_0$. 
If $y_0 = x_0$, we have 
\[\epsilon^{-1} / x_0 + x_0 / \epsilon^{-1} = \epsilon'^{-1} / x_0 + x_0 / \epsilon'^{-1}.\]
Multiplying $\epsilon \epsilon' x_0$ on both side we get
\begin{equation}\label{eq1}
    \epsilon' + \epsilon^2 \epsilon' x^2_0 = \epsilon + \epsilon \epsilon'^2 x^2_0.
\end{equation}
 If $x_0 = 1/y_0$, then we have 
\[\epsilon^{-1} / x_0 + x_0 / \epsilon^{-1} = \epsilon'^{-1}x_0 + 1 / (x_0\epsilon'^{-1}).\]
Multiplying $\epsilon \epsilon' x_0$ on both side we have
\begin{equation}\label{eq2}
    \epsilon' + \epsilon^2\epsilon'x_0^2  = \epsilon x_0^2 + \epsilon \epsilon'^2.
\end{equation}
As the set of such pairs of $(x_0, y_0)$ is infinite, there either exists infinitely many $y_0$ satisfying $x_0 = y_0$ or infinitely many $y_0$ satisfying $x_0 = 1/y_0$. Therefore either there are infinitely many $x_0$ satisfying Equation (\ref{eq1}) or infinitely many $x_0$ satisfying Equation (\ref{eq2}). But if Equation (\ref{eq1}) has infinitely many solutions in $x_0$ then $\epsilon = \epsilon'$ and similarly Equation (\ref{eq2}) having infinitely many solutions implies $\epsilon' = \epsilon^{-1}$. Thus $\epsilon'$ is either $\epsilon$ or $\epsilon^{-1}$. The last statement is clear as the set $\{(x + 1/x , \epsilon^{-1}/x + x / \epsilon^{-1}) : x \in \C^* \}$ is the same as $\{(x + 1/x , \epsilon/x + x / \epsilon) : x \in \C^* \}$ under the map sending $x$ to $1/x$.
\end{proof}

\begin{lem}\label{invcurchey}
    Let K be a number field, $\pi(x) = x + 1/x$, $l(x) = ax + b$ with $a \in \C^*$ and $b \in \C$, and $\epsilon \notin \{\pm 1\}$ be a root of unity. Let $C$ be an irreducible curve in $\P^1_\C \times \P^1_\C$ such that $V(Y - \epsilon X) \subseteq \overline{(\pi,\pi)^{-1}(C)}$. Then $(l, l)(C)$ is given by
    $$
     V(Y^2 + (\epsilon + 1/\epsilon -2)b\cdot (X + Y) - (\epsilon + 1/\epsilon)XY + X^2 + (2 - 1/\epsilon - \epsilon)b^2 + (\epsilon - 1/\epsilon)^2a^2),
    $$ and if it has infinitely many $K$-points
    then $\epsilon + 1/\epsilon \in K$, $a^2 \in K$ and $b \in K$.
\end{lem}
\begin{proof}
    From the condition that $V(Y - \epsilon X) \subset \overline{(\pi, \pi)^{-1}(C)}$ we can explicitly deduce the expression of $C$ since, from the proof of Lemma \ref{pullb}, we know that the Zariski closure of
    $$
    (\pi, \pi)(V(Y - \epsilon X) \cap (\C^* \times \C^*))
    $$
   is $C$. In other words, the irreducible polynomial in two variables vanishing on $$\left\{\left(\frac{1}{x} + x, \frac{1}{\epsilon x} + \epsilon x\right) : \forall x \in \C^*\right\}$$ is the defining equation of $C$. 
  From Lemma \ref{pullb} we know that the defining equation of the irreducible curve $C$ are at least of degree $2$ since both $$(x + 1/x, \epsilon/x + x/\epsilon)$$ and $$(x+ 1/x , \epsilon x + 1/(\epsilon x))$$ are in $C$ and also $\epsilon \notin \{ \pm 1\}$. One can check that 
  \begin{equation}
     C = V(Y^2 - (\epsilon + 1/\epsilon)XY + X^2 + (\epsilon - 1/\epsilon)^2).
  \end{equation}
   Then $(l,l)(C)$ is given by 
    \begin{equation}\label{eq:conic}
        V(Y^2 + (\epsilon + 1/\epsilon -2)b(X + Y) - (\epsilon + 1/\epsilon)XY + X^2 + (2 - \epsilon - 1/\epsilon)b^2 + (\epsilon - 1/\epsilon)^2a^2) .
    \end{equation}
    If $(l,l)(C)$ has infinitely many $K$-points, we have the coefficients of the defining polynomial in Equation (\ref{eq:conic}) are all in $K$. Therefore $$a^2,b \text{ and }\epsilon + 1/\epsilon \in K.$$
\end{proof}

The following lemma is well-known, in fact a similar statement is mentioned in the proof of \cite[Proposition 2.28]{invariantC}, but we are unaware of an explicit reference.
\begin{lem}\label{Cycinv}
Let $C \subset \P^1_\C \times \P^1_\C$ be an irreducible curve over $\C$ that is invariant under either $( P_r, P_r)$ or $(-P_r,-P_r)$. Then either $C = V(X^mY^n - \nu)$ for $m,n \in \Z \setminus \{0\}$ with $\gcd(m,n) = 1$ and $\nu$ a root of unity or $C$ is one of $\{0\} \times \P^1_\C$, $\P^1_\C \times \{0\}$, $\{\infty\} \times \P^1_\C$ or $\P^1_\C \times \{\infty\}$.
\end{lem}
\begin{proof}
We first look at the case that $C$ is an invariant irreducible curve for $(P_r, P_r)$. If $C \cap (\C^* \times \C^*) = \emptyset$, then we have $C$ is one of $ \{0\} \times \P^1_\C $, $\P^1_\C \times \{0\} $, $\{\infty\} \times \P^1_\C$ or $\P^1_\C \times \{\infty\}$. Otherwise, let $W = C \cap (\C^* \times \C^*)$, which is a curve inside $\C^* \times \C^*$ that is invariant under $( P_r,  P_r)$. Since there are only countably many torsion points in $\C^* \times \C^*$ (viewed as an algebraic group with pointwise multiplication), we can take a non-torsion point $(a,b) \in W$ and take $\Gamma \subset \C^* \times \C^*$ to be the subgroup generated by $(a,b)$. Then $\Gamma \cap W $ has infinitely many points as $( P_r^{\circ n}(a),  P_r^{\circ n}(b))\in \Gamma \cap W$ for every $n \in \mathbb{N}$. The Mordell-Lang's conjecture for semi-abelian varieties (proved by Vojta \cite{Vojta96}) says that $\Gamma \cap W = \cup_{1 \leq i \leq k} (\lambda_{i}  H_{i})$ where $\lambda_{i} \in \C^* \times \C^*$, $H_{i}$ are subgroups of $\Gamma$, and $k \geq 0$. Since $\Gamma \cap W$ is dense in $C$ and $C$ is irreducible, there is some $j$ such that $\overline{\lambda_j H_j} = C$. Thus $C = V(X^{m}Y^{n} - \nu)$ for some $m,n \in \Z\setminus \{0\}$ and $\nu \in \C^*$. In fact $\gcd(m,n) =  1$ since otherwise $X^mY^n - \nu$ would be reducible. The fact that $C$ is invariant under the map $( P_r,  P_r)$ implies that $\nu$ is a root of unity. The case for invariant curves under $(-P_r, -P_r)$ is similar and the only difference in the proof is that we take $\Gamma $ generated by $\{(a,b), (-1,-1)\}$.
\end{proof}
\section{Proof of Theorem \ref{main:theorem}}
In this section we complete the proof of the Theorem \ref{main:theorem}. Let's first construct $N \in \mathbb{N}^+$ and $Z \subset \P^1_K \times \P^1_K$.

Let $\Delta_K \subseteq (\mathbb{P}^1_K\times \mathbb{P}^1_K)$ be the diagonal subvariety. By \cite[Lemma 5.7]{DynamicalCrevised}, all the possible irreducible curves over $\mathbb{C}$ different from $\Delta$, having a Zariski dense set of $K$-points and living inside the preimages of $\Delta$ under elements of the form $(f,f)$ with $f \in S$, are in a finite set $\Sigma$. Also, by the proof of \cite[Theorem 1.6]{DynamicalCrevised}, there is a finite set $Z \subset \mathbb{P}^1_K \times \P^1_K$ with the property that if $C \in \Sigma$ and $C'$ is an irreducible component of $(f,f)^{-1}(C)$ not in $\Sigma$ then $C'(K) \subseteq Z$. 
Let 
\begin{equation}
    N = \#\Sigma. 
\end{equation}
Now we show that if there is a point $(a,b)$ outside $Z$ that makes the cancellation Equations (\ref{canceleq1}) and (\ref{canceleq2}) fail for some $(\phi_1, \dots, \phi_k) \in S^k$, with $k > N$, then one of $(1) - (3)$ in the statement of Theorem \ref{main:theorem} must hold.

 \subsection{Producing $h_1$ and an irreducible invariant curve} \label{3.1}\hfill \break
 To be precise, suppose that there exists $k \in \mathbb{N}$, $k > N$, $(a,b) \in (\P^1_K \times \P^1_K) \setminus Z$ and $\phi_{j} \in S$, $j=1,2,3, \dots, k$ such that $$\phi_{k} \circ \dots \circ \phi_{1}(a) = \phi_{k} \circ \dots \circ \phi_{1}(b) $$ and $$\phi_{k-1} \circ \dots \circ \phi_{1}(a) \neq \phi_{k-1} \circ \dots \circ \phi_{1}(b).$$
 \begin{claim}
There exists $k-N-1 < m_1 < m_2 \leq k -1$ and a $C' \in \Sigma$ which is an irreducible component of  $(\phi_{k} \circ \dots \circ \phi_{m_2+1},\phi_{k} \circ \dots \circ \phi_{m_2+1})^{-1}(\Delta)$ such that
\begin{equation}
    (\phi_{m_1 } \circ \dots \circ \phi_{1}(a),\phi_{m_1 } \circ \dots \circ \phi_{1}(b)) \in C'
\end{equation}
 and
 \begin{equation}
     (\phi_{m_2} \circ \dots \circ \phi_{m_1 + 1},\phi_{m_2} \circ \dots \circ \phi_{m_1 + 1})(C') = C'.
 \end{equation}

\end{claim}
\begin{proof}
Since $\phi_{i_{k}} \circ \dots \circ \phi_{i_{1}}(a) = \phi_{i_{k}} \circ \dots \circ \phi_{i_{1}}(b)$ we have 
$$(\phi_{i_{k}} \circ \dots \circ \phi_{i_{1}}(a),\phi_{i_{k}} \circ \dots \circ \phi_{i_{1}}(b)) \in \Delta.$$
We take $C_0 = \Delta$. Then there exists a $C_1 \in \Sigma$ such that $$(\phi_{k-1} \circ \dots \circ \phi_{1}(a), \phi_{k-1} \circ \dots \circ \phi_{1}(b)) \in C_1,$$ $$C_1 \subseteq (\phi_{k}, \phi_{k})^{-1}(C_0).$$ 
In general, for $1 \leq j \leq k$ we can find $C_{k-j} \in \Sigma$ such that $$(\phi_{j} \circ \dots \circ \phi_{i_{1}}(a),\phi_{j} \circ \dots \circ \phi_{1}(b)) \in C_{k-j},$$ $$C_{k-j} \subseteq (\phi_{j+1}, \phi_{j+1})^{-1}(C_{k-j-1})$$ where $1 \leq j < k$. Such curves $C_1, \dots, C_k \in \Sigma$ exist as $(a,b)$ is not in $Z$. However we know that $k > N$ and $N  = \#\Sigma$. Thus there exist $m_1$ and $m_2$ with $k-N-1 < m_1 < m_2 \leq k-1$ such that $C_{k-m_1 } = C_{k- m_2}$. Let $C' = C_{k-m_1} = C_{k - m_2}$. By the construction of $C_i$'s, we have 
$$(\phi_{m_1 } \circ \dots \circ \phi_{i_1}(a),\phi_{m_1 } \circ \dots \circ \phi_{1}(b)) \in C',$$
$$C' \subseteq (\phi_{m_2} \circ \dots \circ \phi_{m_1 + 1 }, \phi_{m_2} \circ \dots \circ \phi_{m_1 + 1})^{-1}(C')$$
In particular, $(\phi_{m_2} \circ \dots \circ \phi_{m_1 + 1}, \phi_{m_2} \circ \dots \circ \phi_{m_1 + 1})(C') \subseteq C'$. Thus $C'$ is an invariant curve of $(\phi_{i_{m_2}} \circ \dots \circ \phi_{m_1 + 1},\phi_{m_2} \circ \dots \circ \phi_{m_1 + 1})$ as the map is a non-constant morphism and hence surjective \cite[243, I §5, Theorem 4]{MR0447223}; i.e., $(\phi_{m_2} \circ \dots \circ \phi_{m_1  + 1},\phi_{m_2} \circ \dots \circ \phi_{m_1 + 1})(C') = C'.$
\end{proof} \hfill \break
We take 
\begin{equation}\label{eq:h_1}
    h_1 = \phi_{m_2} \circ \dots \circ \phi_{m_1  + 1}.
\end{equation}
Then $C'$ is an irreducible curve that is invariant under $(h_1, h_1)$.
\newline
\subsection{ Reduction to cases $(1)-(3)$ in Theorem \ref{main:theorem}} \hfill \break
In this subsection, we continue using the notation from subsection \ref{3.1}. We have already seen that $h_1$ in Equation (\ref{eq:h_1}) can be written as a composition of no more than $N$ elements in $\langle S \rangle$. In this subsection, we want to produce the corresponding $h_2$ and show that $h_2$ can be written as a composition of no more than $N$ elements in $\langle S \rangle$. \newline
Notice that for every $c \in \C$ neither $(\infty, c)$ nor $(c, \infty)$ is in $C'$  since 
$$
\phi_k \circ \dots \circ \phi_1 (\infty) \neq \phi_k \circ \dots \circ \phi_1 (c).
$$
Removing $(\infty, \infty)$ from $C'$, we can view $C'$ as an irreducible curve in $\C \times \C$ and invariant under $(h_1, h_1)$.
Then \cite[Theorem 4.9]{polynomialC} says that if $h_1$ is not linearly conjugate to $P_r$ or $\pm T_r$ then $C'$ is of the form $V(y - p(x))$ or $V(x - p(y))$ for some polynomial $p$ commuting with $h_1$. 
 
 \subsubsection{Case 1:}
  $h_1$ is not linearly conjugate to a cyclic polynomial or to a Chebyshev polynomial. Then, without loss of generality, we assume 
  \begin{equation}
      C' = V(y - p(x)).
  \end{equation}
  Since $C' \subseteq (\phi_{k} \circ \dots \circ \phi_{m_2 + 1}, \phi_{k} \circ \dots \circ \phi_{m_2 + 1})^{-1}(\Delta)$, the polynomial $p$ satisfies 
\[\phi_{k} \circ \dots \circ \phi_{m_2 + 1}(x) =  \phi_{k} \circ \dots \circ \phi_{m_2 + 1}(p(x)).\]
By a simple degree argument, this can only happen when ${\rm deg}(p) = 1$; i.e., $p(x) = sx + t$. Moreover, since $C' \neq \Delta$, we have $(s,t) \neq (1,0)$. Then 
\[\phi_{k} \circ \dots \circ \phi_{m_2 + 1}(x) = \phi_{k} \circ \dots \circ \phi_{m_2 + 1}(sx  + t),\]
and since $h_1$ commutes with $p$, we also have
\[h_1(sx + t) = s h_1(x) + t.\] 
If $s = 1$ and $t \neq 0$, then $\phi_{k} \circ \dots \circ \phi_{m_2 + 1}$ is a constant function, which is impossible. So, $s \neq 1 $ and there exists $l(x) = x + t/(1-s)$ such that $$l \circ h_1 \circ l (sx) =s l^{-1} \circ h_1 \circ l(x). $$ Then $s$ must be a root of unity, which can be seen by looking at the leading term of the polynomials. Therefore
$$l ^{-1}\circ h_1 \circ l(x) = xQ(x^d)$$
$$\phi_{k} \circ \dots \circ \phi_{m_2 + 1} \circ l = P \circ P_d$$ for some polynomials $Q$, $P$ and $d > 1$, where $d$ is the order of $s$ as a root of unity. Since $m_2 > k - N -1$, we can take 
\begin{equation}
    h_2 = \phi_k \circ \dots \circ \phi_{m_2 + 1}
\end{equation}
where $h_2$ is written as a composition of no more than $N$ elements in $S$. Moreover since $$C' =V(y - p(x)) =  V(y - sx - t) \in \Sigma$$ contains a Zariski dense set of $K$-points, $s,t \in K$ and $l(x) = x + t/(1-s)$. Thus we are in case $(1)$ of Theorem \ref{main:theorem}.

Now we are left with the case that $h_1$ is linearly conjugate to $P_r$ or $\pm T_r$.
\subsubsection{Case 2:}
     $h_1$ is linearly conjugate to $P_r$. Without loss of generality we may assume that $h_1 = P_r$ for some $r$. So $C'$ is an irreducible curve invariant under $(P_r, P_r)$. The irreducible curves over $\mathbb{C}$ invariant under power maps are given by $X^mY^n = \nu$ for $m,n \in \Z \setminus \{0\}$, $\gcd(m,n) = 1$ and $\nu$ is a root of unity, or one of $\{0\} \times \P^1_\C$, $\P^1_\C \times \{0\}$, $\{\infty\} \times \P^1_C$ or $\P^1_C \times \{\infty\}$ by Lemma \ref{Cycinv}. Since we have $(\phi_{k} \circ \dots \circ \phi_{m_2 + 1} , \phi_{k} \circ \dots \circ \phi_{m_2 + 1} ) (C') \subseteq \Delta$ by the construction of $C'$, we have the first case holds with $m = -n$ by a simple degree argument. Therefore the irreducible invariant curves are of the form $V(Y - \epsilon X)$ with some root of unity $\epsilon$ satisfying $(\epsilon x)^r = \epsilon x^r$. This implies that the order $d$ of the $\epsilon$ satisfies $d |( r - 1)$. The condition $(\phi_{k} \circ \dots \circ \phi_{m_2 + 1} , \phi_{k} \circ \dots \circ \phi_{m_2 + 1} ) (C') \subseteq \Delta$ and $C' \neq \Delta$ gives that 
     \begin{equation}
         \phi_{k} \circ \dots \circ \phi_{m_2 + 1} = P \circ P_{d}
     \end{equation} for some polynomial $P$ and $d \geq 2$. In general, the conclusion is that there exists complex linear polynomials $l = ux + v$ such that $$\phi_{k} \circ \dots \circ \phi_{m_2 + 1} \circ l = P \circ P_{d}$$ 
$$l^{-1} \circ h_1 \circ l = P_{r}$$ 
 for some polynomial $P$, $d | (r - 1)$ and $d \geq 2$. Similarly, we can take 
\begin{equation}
    h_2 = \phi_k \circ \dots \circ \phi_{m_2 + 1} = P \circ P_d \circ l^{-1}
\end{equation}
where $h_2$ is written as a composition of no more than $N$ elements in $S$. Also \begin{equation}
    C' = (l,l)(V(Y - \epsilon X)) = V(Y - \epsilon X + (\epsilon - 1)v) \in \Sigma,
\end{equation} with some $d$-th root of unity $\epsilon \neq 1$, contains infinitely many $K$-points. This implies $Y - \epsilon X + (\epsilon - 1)v$ is defined over $K$. Thus $\epsilon, v \in K$. Therefore we are in case $(2)$ of Theorem \ref{main:theorem}.

\subsubsection{Case 3 :}
 $h_1$ is linearly conjugate to $\pm T_r$. We may first assume without loss of generality that $h_1$ is either $T_r$ or $-T_r$. Notice that $\pi: \C^* \to \mathbb{A}^1_{\mathbb{C}}$ given by $\pi(x) = 1/x + x$ is dominant and $\pm T_r \circ \pi = \pi \circ ( \pm P_r)$. Therefore the preimages under $(\pi, \pi)$ of invariant subvarieties in $\mathbb{P}^1_\mathbb{C} \times \mathbb{P}^1_\mathbb{C}$ under $(\pm T_r, \pm T_r)$ are invariant subvarieties under $( \pm P_r, \pm P_r)$.   
 Recall from Section \ref{3.1} that $C'$ is an irreducible curve invariant under $(h_1,h_1)$ with Zariski dense set of $K$-points. Now, in this case, $C'$ is an irreducible curve invariant under either $( T_r,  T_r)$ or $(-T_r, -T_r)$ over $\mathbb{C}$, and so 
 \begin{equation}\label{Ecomp}
     E = \overline{(\pi , \pi)^{-1}(C')}
 \end{equation} is an invariant subvariety under either $( P_r,  P_r)$ or $(-P_r , -P_r)$. Let $E_1, \dots , E_t$ be the irreducible components of $E$. Since $E$ is invariant under $(\pm P_r, \pm P_r)$, we have
 $$\overline{(\pm P_r, \pm P_r)(E_i)}  = E_{\tau(i)}$$
  for some map $\tau : \{1,2, \dots, t\} \to \{1,2, \dots, t\}$. Then there exists $j \in \{1,2,\dots, t\}$ and $w \in \mathbb{N}^+$ such that $\tau^w(j) = j$ and therefore 
  $$(\pm P_r , \pm P_r)^{\circ w}(E_j) = E_j,$$
  since $(\pm P_r , \pm P_r)^{\circ w}$ is not a constant morphism and therefore surjective. Now we have that $E_j$ is an irreducible curve invariant under either $(  P_r,  P_r)^{\circ w}$ or $( - P_r, - P_r)^{\circ w}$. 
  
  In both cases, by Lemma \ref{Cycinv}, $E_j$ is given by $X^mY^n = \nu$ for some coprime non-zero integers $m,n$ and a root of unity $\nu$, since it cannot be one of the $\{0\} \times \P^1_\C$, $\P^1_\C \times \{0\}$, $\{\infty\} \times \P^1_\C$ or $\P^1_\C \times \{\infty\}$ as $$(\phi_{k} \circ \dots \circ \phi_{m_2 + 1 },\phi_{k} \circ \dots \circ \phi_{_2 + 1 }) \circ (\pi, \pi)(E_j) \subseteq \Delta.$$ By the same reason, a degree argument implies that $m = -n$.
Therefore 
\begin{equation}\label{Ejcomp}
    E_{j} = V(Y - \epsilon X)
\end{equation} with a root of unity $\epsilon $ of order $d > 1$. The discussion from now on doesn't depend on the sign of $(\pm T_r, \pm T_r)$, so we only consider the case for $(h_1,h_1) = (T_r,T_r)$. Then $E_{j} = V(Y - \epsilon X) \subseteq E$ with the root of unity $\epsilon$ of order $d > 1$ and Lemma \ref{pullb} implies that 
$$V(Y - \epsilon X) \text{ and } V(Y - \epsilon^{-1}X)$$
are the only irreducible components of $E = \overline{(\pi, \pi)^{-1}(C')}$. The fact that $E$ is invariant under $(P_r,P_r)$ implies 
\begin{equation}\label{cheyinveq1}
    (P_r, P_r)(V(Y - \epsilon X)) \subset V(Y - \epsilon X) \cup V(Y - \epsilon^{-1} X). 
\end{equation}
In particular, Equation (\ref{cheyinveq1}) implies $\epsilon^r \in \{ \epsilon^{- 1}, \epsilon\}$. Therefore $$d | (r-1) \text{ or } d | (r +1).$$ On the other hand,  $(\phi_{k} \circ \dots \circ \phi_{m_2 + 1 },\phi_{k} \circ \dots \circ \phi_{m_2 + 1 }) (C') \subseteq \Delta$ implies 
\begin{equation}\label{cheeq}
    \phi_{k} \circ \dots \circ \phi_{m_2 + 1} (\epsilon^{\pm 1}/x + x/\epsilon^{\pm 1}) = \phi_{k} \circ \dots \circ \phi_{m_2 + 1} (1/x + x).
\end{equation}
 Notice that every polynomial can be written as a linear combination of Chebyshev polynomials; in particular, $$\phi_{k} \circ \dots \circ \phi_{m_2 + 1} (x) = \sum^{v}_{i = 1} a_i T_i(x)$$ for some constants $a_i \in \C$, with $v = \deg (\phi_{k} \circ \dots \circ \phi_{m_2 + 1 })$. For $\mu = \epsilon^{\pm 1}$ we have 
\begin{equation}
    \phi_{k} \circ \dots \circ \phi_{m_2 + 1} (\mu/x + x/ \mu) = \sum^{v}_{i= 1} a_i (\mu^i/x^i + x^i / \mu^i) = \sum^{v}_{i= 1} a_i (1/x^i + x^i )
\end{equation} by Equation (\ref{cheeq}) for all $x \in \C^*$. So for each $i \in \{1, 2, \dots, v\}$ such that $a_i \neq 0$, we have $\mu^{i} = 1$; i.e., $d | i$. In particular $d | v$. Therefore 
$$\phi_{k} \circ \dots \circ \phi_{m_2 + 1} (x) = \sum^{v/d}_{j = 1 } a_{jd} T_{j} \circ T_{d} = P \circ T_{d}$$
for some polynomial $P$.   
Therefore, in general we have that there exists a complex linear polynomial $l(x)= ux + v$ such that $$ \phi_{k} \circ \dots \circ \phi_{m_2 + 1} \circ l = P \circ T_{d},$$ 
$$l^{-1} \circ h_1 \circ l = \pm T_r$$ where $d| (r + 1) $ or $d| (r - 1)$ and $P$ is an arbitrary polynomial. Similarly, we can take 
\begin{equation}
    h_2 = \phi_k \circ \dots \circ \phi_{m_2 + 1}
\end{equation}
where $h_2$ is written as a composition of no more than $N$ elements in $S$. Also Equation (\ref{Ecomp}) and (\ref{Ejcomp}) implies that \begin{equation}\label{curcondchey}
    V(Y - \epsilon X) \subset \overline{(\pi, \pi)^{-1}((l,l)^{-1}(C'))}
\end{equation} with some root of unity $\epsilon \neq 1$ of order $d$, where $C' \in \Sigma$ has infinitely many $K$-points. If $d = 2$ then in this case from Equation (\ref{curcondchey})
$$C' = V(Y + X - 2v) = (l,l)(\overline{(\pi,\pi)(V(Y + X))}) \in \Sigma.$$
In particular, $C'$ has infinitely many $K$-points is equivalent to $v \in K$. This is $3(a)$ of Theorem \ref{main:theorem}.
If $d >2$ then by Lemma \ref{invcurchey} and Equation (\ref{curcondchey}), we have that
\begin{equation} \label{inveq0}
    C' = V(Y^2 + (\epsilon + 1/\epsilon -2)v(X + Y) + (\epsilon - 1/\epsilon)XY + X^2 + (2 - \epsilon - 1/\epsilon)v^2 + (\epsilon - 1/\epsilon)^2u^2).
\end{equation}
has infinitely many $K$-points. Since this is a conic and in particular a genus zero curve, so it has infinitely many $K$-points if and only if there exists one. This is $3(b)$ of Theorem \ref{main:theorem}. If we further have that $r$ is even, then 
$$ l \circ \pm T_r \circ l^{-1} \in K[x]$$
implies that $u^{r - 1} \in K$. From Lemma \ref{invcurchey} and recall $C' \in \Sigma$, we have $$u^2, v, \epsilon + 1/\epsilon \in K.$$ So, $u \in K$. Notice that 
$$(2u + v, u(\epsilon + 1/\epsilon) + v ) \in K^2$$
is a point in $C'$. Therefore, $C'$ has $K$-points if and only if $u,v, \epsilon + 1/\epsilon \in K$.

 \subsection{Sufficiency of $(1) - (3)$ in Theorem \ref{main:theorem}} \hfill \break
 For this we just need to verify the moreover statement of Theorem \ref{main:theorem}. We suppose there exist $h_1$, $h_2$ and $l$ as in cases $(1) - (3)$ in the statement of Theorem \ref{main:theorem}. The idea is that from previous sections we have already obtained the invariant curves under $(h_1,h_1)$ and then the set of $K$-points on the invariant curves excluding the preimages of its intersection with $\Delta$ under some iterations of $(h_1,h_1)$ will satisfy Equation (\ref{eq: violate-cancel1}) and (\ref{eq: violate-cancel2}) in Theorem \ref{main:theorem}. We will make these precise as follows.
 \newline
 \newline
  If conditions in $(1)$ or $(2)$ of Theorem \ref{main:theorem} hold, then take $$C = V(Y - \epsilon X - v(1- \epsilon))$$ for some primitive $d$-th root of unity $\epsilon  \in K$. By the argument in the previous section or one can check directly that $$ (h_1,h_1)(C) \subseteq C.$$ Let $j$ be a nonnegative integer and let
 $$
 Z_1 = (h_1,h_1)^{-j}(C \cap \Delta) \cap C 
 .$$
 Notice that $Z_1$ is a finite set. So, we take $(a,b) \in C(K) \setminus Z_1$ and then 
 $$
 h^j_1(a) \neq h^j_1(b),
 $$
 $$(h^j_1(a), h^j_1(b)) \in C$$
 and also 
 $$
 h_2 \circ h^j_1(a) = h_2 \circ h^j_1(b)
 $$
by the construction of $h_2$. 
\newline
\newline
 For case $(3)$ in Theorem \ref{main:theorem}, let $\epsilon $ be a primitive $d$-th root of unity and let $C$ be
 $$ 
   V(Y^2 + (\epsilon + 1/\epsilon -2)v(X + Y) + (\epsilon + 1/\epsilon)XY + X^2 + (2 - \epsilon - 1/\epsilon)v^2 + (\epsilon - 1/\epsilon)^2u^2)$$ if $3(b)$ of Theorem \ref{main:theorem} holds; or 
    $$C = V(Y + X -2v)$$ if $3(a)$ of Theorem \ref{main:theorem} holds.
 Let $j$ be a nonnegative integer and let 
 $$
 Z_1 = (h_1,h_1)^{-j}(C \cap \Delta) \cap C,
 $$
 which is a finite set.
 Then for any $K$-points $(a,b) \in C(K) \setminus Z_1$ we have $$h^{j}_1 (a) \neq h^j_1(b).$$ Recall from the argument in the last subsection, $C$ is invariant under $(h_1,h_1)$ and $(h_2,h_2)(C) \subset \Delta$. Therefore 
 $$
 h_2\circ h^j_1(a) = h_2 \circ h^j_1(b).
 $$
 
  This concludes the proof.
  
\section{Effectively Computing $N$ in Theorem \ref{main:theorem}}
It is natural to ask if the number $N$ and the number of points in $Z$ in Theorem \ref{main:theorem} can be explicitly bounded. Bounding $\# Z$ leads us to look at the $K$-points of genus $\geq 2$ curves appear in the preimages of $\Delta \subseteq \P^1 \times \P^1$ under maps $(f,f)$ with $f \in \langle S \rangle$. It is still an open problem to bound the number of $K$-points on curves uniformly in terms of the number field $K$ and the genus of the curve \cite{CHM22}. Also, it is difficult to effectively bound the number of $K$-points on a given curve in general. So it is not an easy task to bound $\# Z$. But we can indeed get a bound on $N$:
\begin{cor}\label{Nbound}
Adopt the notation from Theorem \ref{main:theorem}. There is an explicitly computable number $N_0 \in \mathbb{N}$ depending on $S$ such that we can take $N = N_0$ in the statement of Theorem \ref{main:theorem} . 
\end{cor}
\begin{proof}
Using the polynomial decomposition algorithm in \cite{polynomialDec}, we can obtain a decomposition of the elements of $S= \{ \phi_1, \dots, \phi_r\}$ as a composition of indecomposable polynomials and a number field $L$ over which each indecomposable factor is defined. For each $\phi_i \in S$, let $\phi_i = f_{i,1} \circ \dots \circ f_{i,{r_i}}$ be such a decomposition with each $f_{i,j}$ indecomposable. Consider the embedding $$i_{\mathcal{O}_{(1,1)}} : \mathbb{P}^1_L \times \mathbb{P}^1_L  \to \mathbb{P}^3_L $$ $$i_{\mathcal{O}_{(1,1)}}((x_1 : x_2),(y_1:y_2)) = (x_1y_1:x_1y_2:x_2y_1:x_2y_2),$$ where $(x_1 : x_2)$ and $(y_1 : y_2)$ are homogeneous coordinates of the two factors of $\P^1_L \times \P^1_L$. Then for each $f_{i,j}$ there exists a morphism $F_{i,j} : \mathbb{P}^{3}_L \to \mathbb{P}^{3}_L$ such that $i_{\mathcal{O}_{(1,1)}} \circ (f_{i,j} , f_{i,j}) = F_{i,j} \circ i_{\mathcal{O}_{(1,1)}} $. We consider the height function $h$ and the canonical height function $\hat{h}_{f_{i,j}}$ of the subvarieties in $\mathbb{P}^1_L \times \mathbb{P}^1_L$ obtained by first embedding the subvarieties in $\mathbb{P}^3_L$ through $i_{\mathcal{O}_{(1,1)}}$ and then using the definition of height and the canonical height function of the subvarieties in $\P^3_L$ \cite[Definition 4.1, Definition 4.8]{goodR}. Let $\Sigma$ be the same set of irreducible curves as in the proof of Theorem \ref{main:theorem} and notice that we also showed $N$ can be chosen as $\#\Sigma$ in the proof. By the nature of the statement of Theorem \ref{main:theorem}, every $N \geq \#\Sigma$ will work. 
Let $$d_{i,j} = \deg (f_{i,j})$$ and let $\deg_{\mathcal{O}_{(1,1)}} (W)$ be the degree of a subvariety $W \subset \P^1 \times \P^1$ defined over $L$ with respect to the embedding $i_{\mathcal{O}_{(1,1)}}$. We take $M, \Tilde{C} \in \R$ such that
\begin{equation}
    \tilde{C} = {\rm max} \left\{\frac{4 C(f_{i,j},3,4)}{(d_{i,j} - 1)} \right\}_{i,j},
\end{equation}
where $C(f_{i,j},3,4)$ is an explicitly computed constant in \cite[Theorem 1.1]{goodR} satisfying  
\[|\hat{h}_{f_{i,j}}(W) - h(W)|\leq \frac{4 C(f_{i,j},3,4)}{(d_{i,j} - 1)} \]
for every subvariety $W \subset \P^1 \times \P^1$ defined over $L$ with ${\rm deg}_{\mathcal{O}_{(1,1)}}(W) \leq 4$, and $M$ satisfies 
\begin{itemize}
    \item $h(\Delta) < M$
    \item $ \max\left(\frac{d_{i,j}^2 - 1}{d^2_{i,j}}\right)< (1 - \frac{2\tilde{C}}{M})^2.$
\end{itemize}
Notice that this $\Tilde{C}$ by our construction satisfies
\begin{equation}
    |\hat{h}_{f_{i,j}}(W) - h(W)| \leq \tilde{C}
\end{equation}
for every $f_{i,j}$ appearing in the decomposition of some $\phi_i \in S$ and $W \subset \P^1 \times \P^1$ defined over $L$ with ${\rm deg}_{\mathcal{O}_{(1,1)}}(W) \leq 4$. We take $\Sigma'$ to be the set of irreducible curves $W \subset \P^1 \times \P^1$ defined over $L$ such that 
$$
 \deg_{\mathcal{O}_{(1,1)}}(W) \leq 4,
$$ 
and
$$
 h(W) < M/\left(1 - \frac{2\tilde{C}}{M}\right).
$$
Then, by the argument in \cite[Section 5]{DynamicalCrevised}, we have $\Sigma \subseteq \Sigma'$.
  Since $\# \Sigma'$ is bounded by an explicitly computable number $N_0$ as it is a set of subvarieties with bounded heights and degrees, where the bounds are explicitly computable, we can take $N = N_0 \geq \#\Sigma$.
\end{proof}

\section*{Acknowledgments}
We thank Jason P. Bell for helpful discussions.
\bibliographystyle{amsplain}
\bibliography{Dynamical-Cancellation-of-Polynomials}
\end{document}